\documentclass{article}[12pt] 

\usepackage{amsmath}
\usepackage{amsthm}
\usepackage{amsfonts}
\usepackage{mathrsfs}
\usepackage{stmaryrd}
\usepackage{setspace}
\usepackage{fullpage}
\usepackage{amssymb}
\usepackage{breqn}
\usepackage{enumitem}
\usepackage{bbold} 
\usepackage{authblk}
\usepackage{comment}
\usepackage{hyperref}
\usepackage{pgf,tikz}
\usepackage{graphicx}
\usepackage{subcaption}

\newtheorem{thm}{Theorem}[section]

\newtheorem{conjecture}[thm]{Conjecture}
\newtheorem{proposition}[thm]{Proposition}

\newtheorem*{theorem*}{Theorem}
\newtheorem*{prop*}{Proposition}

\newcommand\cF{{\mathcal F}}
\newcommand\cG{{\mathcal G}}
\newcommand\cH{{\mathcal H}}

\usepackage{lineno}

\newcommand{\ignore}[1]{}

\title{A note on hyperseparating set systems}
\author{D\'aniel Gerbner\footnote{HUN-REN Alfr\'ed R\'enyi Institute of Mathematics, E-mail: \texttt{gerbner@renyi.hu.}}}
\date{}

\begin{document}

\maketitle

\begin{abstract}
We say that a set system $\cF$ is $k$-completely hyperseparating if for any vertex $v$, there are at most $k$ sets in $\cF$ with intersection $\{v\}$. We determine the minimum size of such set systems on an $n$-element underlying set, generalizing a very recent result for $k=2$ by Bat\'ikov\'a, Kepka, and Nem\u{e}c.

We say that $\cF$ is $k$-hyperseparating if for any vertex $v$, there are at most $k$ sets in $\cF$ such that no other vertex is contained by exactly the same sets out of these $k$ sets. We determine the minimum size of $2$-hyperseparating set systems on an $n$-element underlying set.
\end{abstract}

\section{Introduction}

A set system $\cF$ on an underlying set $V$ is called \textit{separating} if for any elements $v\neq v'$ of $V$, there is a set in $\cF$ that contains exactly one of them. This is a central object in the topic Search Theory/Group Testing (see \cite{DH} for a monograph), which seeks to find one (or more) hidden elements (usually called \textit{defective}) in $V$ by asking queries of the following form for some $F\subset V$: ``Is the defective element contained in $F$''? Asking all the sets in a family $\cF$ always finds the defective element if and only if $\cF$ is separating. Therefore, the goal is to minimize the size of separating set systems. 

It is well-known and easy to see that if $|V|=n$, then the smallest size of a separating set system is $\lceil\log_2 n\rceil$. Indeed, a simple construction is obtained by mapping the elements to $\{1,\dots, n\}$ and taking the sets $V_i:=\{j\le n: \text{the $i$th digit of $j$ is 1}\}$. On the other hand, if we list the sets that contain a given $v\in V$, we obtain a different list for different elements. There are $2^{|\cF|}$ possible lists, thus $2^{|\cF|}\ge n$.

A natural variant is the so-called \textit{completely separating set system}, where for any elements $v\neq v'$ of $V$, there is a set in $\cF$ that contains $v$ but not $v'$ (note that there is also a set that contains $v'$ but not $v$ by exchanging the role of $v$ and $v'$). This notion was introduced by Dickson \cite{dickson}. Spencer \cite{spencer} proved that the smallest size of a separating set system is $\min\{m: \binom{m}{\lfloor m/2\rfloor}\ge n\}$. Note that completely separating set systems did not originally have a motivation from Search Theory, but later they also appeared in that area \cite{gv,gknpw}.

Very recently, Bat\'ikov\'a, Kepka and Nem\u{e}c \cite{bkn} introduced \textit{hypercompletely separating systems}. They are set systems with the property that for every element $v$ of the underlying set $V$, there are sets $A,B$ in the set system with $A\cap B=\{v\}$. Note that $A$ and $B$ are not necessarily distinct. Bat\'ikov\'a, Kepka and Nem\u{e}c \cite{bkn} determined the minimum size of a hypercompletely separating system on an $n$-element underlying set.

In this note, we consider two generalizations. We say that a set system $\cF\subset 2^{V}$ is \textit{$k$-hypercompletely separating} if for any element $v\in V$, there are (not necessarily distinct) sets $A_1,\dots, A_k\in \cF$ such that $A_1\cap A_2\cap\dots\cap A_k=\{v\}$. Note that the case $k=2$ gives back the hypercompletely separating systems from \cite{bkn}.

One can rephrase the definition of completely separating set systems in the following way. For any element $v\in V$, the intersection of sets containing $v$ is $v$. In other words, there are some sets in $\cF$ such that their intersection is $v$. Here, we assume that there are $k$ or fewer sets in $\cF$ such that their intersection is $v$.

We determine the smallest size of a $k$-hypercompletely separating set system for every $k$. Let $k'=k'(m)$ be $k$ if $m\ge 2k-1$ and $\lfloor m/2\rfloor$ if $m\le 2k-1$.

\begin{proposition}\label{prop1}
    The smallest size of a $k$-hypercompletely separating set system on an $n$-element underlying set is $\min \{m:\binom{m}{k'}\ge n\}$.
\end{proposition}

We also consider the analogous restriction for separating set systems. Observe that one can rephrase the definition of completely separating set systems the following way. For any element $v\in V$, there are some sets in $\cF$ such that knowing which ones of those sets contain $v$ completely determines $v$.

We say that a set system $\cF\subset 2^{V}$ is \textit{$k$-hyperseparating} if for any element $v\in V$, there are sets $A_1,\dots, A_k\in \cF$ such that $v\in A_1,\dots, A_i$, $v\not\in A_{i+1},\dots, A_k$, and no other element of the underlying set satisfies this property.

Looking at the search theory motivation, we can assume that we want to determine the defective element, and then we want to provide at most $ k$ sets as witnesses. This can happen if witnesses are important but have a high cost.
Let $f(n,k)$ denote the smallest size of a $k$-hyperseparating set system on an $n$-element underlying set.

\begin{proposition}\label{trivi}
    If $n> \binom{2k-1}{k}$, then we have $\min \{m:2^{k}\binom{m}{k}\ge n\}\le f(n,k)\le \min \{m:\binom{m}{k}\ge n\}$.
\end{proposition}


We conjecture that the upper bound is sharp for large $n$, i.e., there is no smaller construction than the construction for $k$-hypercompletely separating sets systems.

\begin{conjecture} For any $k$, if $n$ is sufficiently large, then
    we have $f(n,k)= \min \{m:\binom{m}{k}\ge n\}$.
\end{conjecture}

We can prove this conjecture for $k=2$.

\begin{thm}\label{tetel}
    We have \begin{displaymath}
f(n,2)=
\left\{ \begin{array}{l l}
 \min \{m:\binom{m}{2}\ge n\} & \textrm{if\/ $n\ge 10$},\\
\lceil n/2\rceil & \textrm{if\/ $n\le 10$}.\\
\end{array}
\right.
\end{displaymath}
\end{thm}

\section{Proofs}

Given a set system $\cF$ on $V$ with $|\cF|=m$, the \textit{dual (multi)set system} $\cF'$ is defined on the underlying set $\cF$. Each $v\in V$ defines a set $F_v:=\{F\in \cF: v\in F\}$ in $\cF'$, thus $\cF'$ consists of $n$ sets on an $m$-element underlying set. Observe that the dual of the dual system is isomorphic to the original set system.

A standard method in this area is to study what properties of $\cF'$ are implied by the properties of $\cF$. For example, $\cF$ is separating if and only if for any two vertices, there is a set that contains exactly one of them. This is equivalent to the property that in $\cF'$, for any two sets, there is an element that is contained in exactly one of them. This means that the sets are distinct, i.e., there are no multisets. In particular, there are at most $2^m$ sets in $\cF'$, i.e., $n\le 2^m$, hence $m\ge \lceil \log n\rceil$. Since the $2^m$ upper bound is sharp, the $m\ge \lceil \log n\rceil$ lower bound is also sharp. In general, determining the minimum size of set systems satisfying some properties is equivalent to determining the maximum size of the dual set systems.

Similarly, $\cF$ is completely separating if and only if in $\cF'$, for every two sets $A,B$ there are elements in $A\setminus B$ and in $B\setminus A$. In other words, no set in $\cF'$ contains another set in $\cF'$. Set systems with this property are called \textit{Sperner} or \textit{antichain}, and their maximum size was determined by Sperner \cite{S}.

Similarly, $\cF$ is $k$-hypercompletely separating if and only if for every set $F\in\cF'$ there is a set $S$ of size at most $k$ that is a subset of $F$, but not a subset of any other member of $\cF'$. Moreover, $\cF$ is $k$-hyperseparating if and only if for every set $F\in\cF'$ there is a set $S$ of size at most $k$ such that $F\cap S\neq S\cap F'$ for any $F'\in \cF'$ with $F'\neq F$. In both cases, the set $S$ is called a \textit{separator}. Let $S'=S'(F)=S\cap F$. We call $S'$ the \textit{key} of $F$. 

Let us start with the proof of Proposition \ref{prop1}. It states that the smallest size of a $k$-hypercompletely separating set system on an $n$-element underlying set is $\min \{m:\binom{m}{k'}\ge n\}$. 

\begin{proof}[Proof of Proposition \ref{prop1}]
    Consider the dual set system $\cF'$ of a $k$-hypercompletely separating set system $\cF$. Our goal is now to determine the largest size of such set systems.

    If $F\in \cF'$ has size larger than $k$, then we can replace $F$ by $S$. Clearly, $S$ was not in $\cF'$, since $F$ is the only member of $\cF'$ that contains $S$. Therefore, the size of $\cF'$ does not change by this replacement, and for the new set $S$, $S$ will be a separator. Therefore, we can assume that each set in $\cF'$ has size at most $k$. Also, $\cF'$ must be a Sperner family, since a subset of $F$ does not have any unique subgraph $S$. If $m\le 2k-1$, then Sperner's theorem \cite{S} shows that $|\cF'|\le \binom{m}{k'}$. If $m\ge 2k-1$, then it is well-known that a Sperner family consisting of sets of size at most $k$ has at most $\binom{m}{k}$ sets. 
    
    This can be seen, for example, by the following argument. If there is a set of size less than $k$ in $\cF'$, then we consider the subfamily $\cG$ of $\cF'$ consisting of the sets of the smallest size $\ell$. Let $\cG'$ consist of all the $\ell+1$-element subsets of the underlying set that contain at least one member of $\cG$. Each set in $\cG$ is contained in $m-\ell$ sets of $\cG'$, and each set in $\cG'$ contains at most $\ell+1<n-\ell$ sets of $\cG$, thus $|\cG'|>|\cG|$. We delete $\cG$ and add $\cG'$ to $\cF'$ in order to obtain a family $\cF''$. Each set in $\cF''$ has size at most $k$, and $\cF''$ is a Sperner family. Indeed, the new sets of $\cF''$ do not contain any other set of $\cF''$, since they have the smallest size in $\cF''$. If $G\in \cG'$ is contained in $F\in \cF''$, then $|F|>|G|$, thus $F\not\in \cG'$. This implies that $F\in \cF'$. But $F$ contains a set in $\cG'$, which contains a set in $\cG$, thus $\cF'$ is not a Sperner family, a contradiction. Therefore, $\cG''$ is indeed a Sperner family consisting of sets of size at most $k$, and is larger than $\cF$, a contradiction. 

We obtained that $n=|\cF'|\le \binom{m}{k'}$. This bound is sharp, as shown by the set system consisting of all the $k'$-element sets. To obtain the bound on $|\cF|$, observe that we showed that $\binom{m}{k'}\ge n$. We can pick the smallest $m$ satisfying this inequality, take the $k'$-element sets on an $m$-element underlying set, and take its dual to obtain a $k$-hyperseparating family of the desired size.
    \end{proof}

Let us continue with the proof of Proposition \ref{trivi} that states that $\min \{m:2^{k}\binom{m}{k}\ge n\}\le f(n,k)\le \min \{m:\binom{m}{k}\ge n\}$.

\begin{proof}[Proof of Proposition \ref{trivi}]
    Consider a $k$-hyperseparating set system $\cF$ on an $n$-element set $V$, and its dual $\cF'$. 
Then $\cF'$ has the property that for any set $F\in \cF'$, there is a set $S$ of size at most $k$ such that $F\cap S\neq S\cap F'$ for any $F'\in \cF'$ with $F'\neq F$.

The upper bound follows from Proposition \ref{prop1}. For the lower bound, observe first that a set of size $i$ can be the separator of at most $2^{i}$ sets. This shows that $|\cF'|\le \sum_{i=0}^k 2^{i}\binom{n}{i}$. 

Now, we are given a family $\cH$ of pairs $(S,S')$ of separators and keys, where $|S|\le k$ and $S'\subset S$. Furthermore, we claim that for any pairs with the same key, i.e., $(S_1,S')$ and $(S_2,S')$, we have that $S_1$ is not a subset of $S_2$. Indeed, otherwise let $(S_i,S')$ belong to the set $\cF_i\in \cF'$, i.e., $F_i\cap S_i=S'$ and $F_i$ is the only member of $\cF'$ with this property. Then $S'\subseteq F_1\cup S_1\subseteq F_2\cap S_1\subseteq F_2\cap S_2=S'$. This shows that $F_1$ is not the only set in $\cF'$ that intersects $S_1$ in $S'$, a contradiction.

Consider now families $\cH$ of pairs $(S,S')$ on an $m$-element underlying set, where $|S|\le k$ and $S'\subset S$, such that for each $S'$, the the sets $S$ with $(S,S')\in \cH$ form a Sperner family. We claim that such families consist of at most $2^k\binom{m}{k}$ pairs. To emphasize that we are dealing with abstract families satisfying the above conditions, instead of separator and key, we say that $S$ is the \textit{large set} of the pair and $S'$ is the \textit{small set}.

Assume that $\cH$ has the maximum number of pairs. If all the large sets form a Sperner family, then there are at most $\binom{n}{k}$ large sets by an argument inside the proof of Proposition \ref{prop1}, which implies the desired upper bound $2^k\binom{n}{k}$. This is not the property we have here, but we can use a similar argument. Note that the threshold on $n$ means that $m\ge 2k$.

 If there is a large set of size less than $k$, then we consider the subfamily $\cG$ of $\cH$ consisting of the large sets of the smallest size $\ell$, together with the corresponding small sets. Let $\cG'$ consist of all the $\ell+1$-element subsets of the underlying set that contain at least one member of $\cG$. Each set in $\cG$ is contained in $m-\ell$ sets of $\cG'$, and each set in $\cG'$ contains at most $\ell+1<m-\ell$ sets of $\cG$, thus $|\cG'|>|\cG|$. For each set in $\cG'$, we add each $S'$ as a small set if $S'$ is the small set of a member of $\cG$. Then $\cG'$ is disjoint from $\cH$ and $\cG'$ contains more pairs than $\cG$. We delete $\cG$ and add $\cG'$ to $\cH$ in order to obtain a family $\cH'$. 
 
 Each set in $\cH'$ has size at most $k$. Indeed, the new sets in $\cH'$ do not contain any other set of $\cH'$, since they have the smallest size in $\cH'$. 
 We have to show that for each small set $S'$, the corresponding large sets form a Sperner family. Indeed, a new large set $S$ cannot be the subset of a new large set, since they have the same size, and cannot be the subset of an old large set $S_1$ where $S_1,S'$ is in $\cH$. $S$ also cannot contain any member of $\cH'$, since $S$ has the smallest size in $\cH'$.

We obtained that $\cH'$ satisfies our assumptions on the family, thus $\cH$ did not contain the largest number of pairs, a contradiction, completing the proof.
\end{proof}

Let us continue with the proof of Theorem \ref{tetel}, which states that $f(n,2)$ is the smaller of $\min \{m:\binom{m}{2}\ge n\}$ and $n/2$. 

\begin{proof}[Proof of Theorem \ref{tetel}] Consider a $2$-hyperseparating set system $\cF$ on an $n$-element set $V$, and its dual $\cF'$. 
Then $\cF'$ has the property that for any set $F\in \cF'$, there is a set $S$ of size at most 2 such that $F\cap S\neq S\cap F'$ for any $F'\in \cF'$ with $F'\neq F$. We call such families \textit{nice}.
Our goal is to determine the largest size of nice families on an $m$-element underlying set, which we denote by $g(m)$. We need to show that $g(m)$ is at most the larger of $\binom{m}{2}$ and $2m$.

Given a family $\cF'$ and an element $v$, we say that we \textit{switch} $v$ if we remove it from the sets in $\cF'$ that contain $v$ and add it to those that do not contain $v$. In other words, we obtain the family
$\{F\cup\{v\}: v\not\in F,\, F\in \cF'\}\cup \{F\setminus \{v\}: v\in F\in\cF'\}$. 
Observe that for a nice family $\cF'$, we can switch any element to obtain another nice family. Indeed, the same $S$ will be a separator.
 
Let us consider two vertices $x,y$ with the property that $\{x,y\}$ is the separator for at least two sets $F,F'\in\cF'$. There are two cases, either $S'(F)$ and $S'(F')$ differ by one element or by two elements. By switching elements, we can assume that either $S'(F)=\{x,y\}$ and $S'(F')=\{x\}$, or $S'(F)=\{x,y\}$ and $S'(F')=\emptyset$.

In the first case, among the sets that contain $x$, only one contains $y$, and only one does not contain $y$. Let us delete $x,F,F'$. Then the resulting set system $\cG$ on $m-1$ vertices has the property that for any set, we have a separator of size 2. Note that the separator here could contain $x$, hence it is not obvious that $\cG$ is a nice family on $m-1$ vertices. However, in this particular case, it is easy to see that $x$ does not actually separate sets in $\cG$, i.e., we could delete $x$ from any separator to obtain another separator in $\cG$. This shows that $\cG$ is nice, thus $\cF'$ has size at most $g(m-1)+2$.


In the second case, we delete $x,y,F,F'$ and add a new element $z$. We add $z$ to each set that contains $x$ but not $y$. Let $\cG$ denote the resulting set system. We show that $\cG$ is nice. Consider $G\in \cG$. By switching, we can assume that $G$ does not contain $z$, i.e., $G\in \cF'$, $x\not\in G$ and $y\in G$. Consider the separator $S(G)$ in $\cF'$ (note that $G$ may have a different separator in $\cG$). If $x,y\not\in S(G)$, then $S(G)$ is still a separator in $\cG$. 

If $S(G)=\{x,w\}$ for some $w\neq y$, that means that $G$ is the only set in $\cF'$ that does not contain $x$ and contains $w$ if and only if $w$ is in the key $S'(G)$. Then $G$ is the only set in $\cG$ that does not contain $z$ and contains $w$ if and only if $v$ is in the key $S'(G)$. Indeed, if $G'\in \cG$ also does not contain $z$ and contains $w$ if and only if $w$ is in the key $S'(G)$, then $G'$ does not contain $x$, hence in $\cF'$ intersects $S(G)$ in $S'(G)$, a contradiction. 

If $S(G)=\{y,w\}$ for some $w\neq x$, that means that $G$ is the only set in $\cF'$ that contains $y$ and contains $w$ if and only if $w$ is in the key $S'(G)$. Then $G$ is the only set in $\cG$ that does not contain $z$ and contains $w$ if and only if $v$ is in the key $S'(G)$. Indeed, if $G'\in \cG$ also does not contain $z$ and contains $w$ if and only if $w$ is in the key $S'(G)$, then $G'$ does not contain $x$, hence in $\cF'$ intersects $S(G)$ in $S'(G)$, a contradiction. 

Finally, if $S(G)=\{x,y\}$, that means that $G$ is the only set in $\cF'$ that does not contain $x$ and contains $y$, but then we are in the first case.

We have obtained that if a 2-set is the separator for more than one member of $\cF$, then $g(m)\le g(m-1)+2$. Otherwise, each 2-set is a separator at most once, and each 1-set is a separator at most twice. Assume that a 1-set is a separator $S$. Let $S=\{x\}$, and by switching we can assume that $S'=S$, i.e., exactly one set $F\in \cF'$ contains $x$. Then we delete $F$ and $x$, and in the rest we obtain a nice set system on $m-1$ vertices, thus $\cF'$ contains at most $g(m-1)+1$ sets.

We have obtained that either $|\cF'|\le \binom{m}{2}$, or $|\cF'|\le g(m-1)+2$. 
Clearly, $g(1)\le 2$ and $g(2)\le 4$, thus $g(3)\le \max\{6,\binom{3}{2}\}=6$, $g(4)\le \max\{8,\binom{4}{2}\}=8$ and $g(5)\le \max\{10,\binom{5}{2}\}=10$. This shows that $g(5)=10$ and for $m> 5$, we can see by induction that $g(m)\le\binom{m}{2}$. Indeed, either we do have this bound immediately, or we have that $g(m)\le g(m-1)+2\le\binom{m-1}{2}+2<\binom{m}{2}$. 

Let us show some constructions for $m\le 3$.
For $m=1$, the two sets are $\emptyset$ and the full set $\{v\}$, and for both of these sets $F$ we have $S=S'=F$. For $m=2$, we have all four sets in $\cF$, and for each set $F$ we have $S=V$, $S'=F$. For $m=3$, we take all the 1-element and 2-element sets. If $|F|=2$, then $S=S'=F$, while if $|F|=1$, then $S'=\emptyset$ and $S$ is the complement of $F$.

Finally, consider the case $m=4$. A simple construction of 8 sets was found by Claude Opus 4.6. The sets in the underlying set $\{1,2,3,4\}$ are $\emptyset$ with witness $\{1,2\}$, $\{1\}$ with witness $\{1,3\}$, $\{2\}$ with witness $\{2,4\}$, $\{1,3\}$ with witness $\{3,4\}$, $\{2,4\}$ with witness $\{3,4\}$, $\{1,3,4\}$ with witness $\{2,4\}$, $\{2,3,4\}$ with witness $\{1,3\}$, and $\{1,2,3,4\}$ with witness $\{1,2\}$.
\end{proof}


\bigskip

\textbf{Declaration of AI usage. }
The construction for $m=4$ in the proof of Theorem \ref{tetel} was found using a script written by Claude Opus 4.6, and was subsequently checked by the author.

\bigskip

\textbf{Funding.} Research supported by the National Research, Development and Innovation Office - NKFIH under the grant KKP-133819 and by the J\'anos Bolyai scholarship.

\end{document}